\documentclass[a4paper]{amsart}
\usepackage{amscd}
\usepackage{amsthm}
\usepackage{graphicx}
\usepackage{amsmath}
\usepackage{amsfonts}
\usepackage{amssymb}
\newtheorem{theorem}{Theorem}

\newtheorem{corollary}[theorem]{Corollary}

\newtheorem{lemma}[theorem]{Lemma}

\newtheorem{proposition}[theorem]{Proposition}
\newtheorem{remark}[theorem]{Remark}

\begin{document}
\title[A NOTE ON HARDY SPACES AND BOUNDED OPERATORS]{ A NOTE ON HARDY SPACES AND BOUNDED OPERATORS}
\author{Pablo Rocha}
\address{Universidad Nacional del Sur, INMABB (Conicet), Bah\'{\i}a Blanca, 8000 Buenos Aires, Argentina}
\email{pablo.rocha@uns.edu.ar}
\thanks{\textbf{Key
words and phrases}: Hardy Spaces, Atomic decomposition.}
\thanks{\textbf{2.010
Math. Subject Classification}: }
\thanks{}
\maketitle
\begin{abstract}
In this note we show that if $f \in H^{p}(\mathbb{R}^{n}) \cap L^{s}(\mathbb{R}^{n})$, where $0 < p \leq 1 < s < \infty$, then there exists a
$(p, \infty)$-atomic decomposition which converges to $f$ in $L^{s}(\mathbb{R}^{n})$. From this fact, we prove that a bounded operator $T$ on $L^{s}(\mathbb{R}^{n})$ can be extended to a bounded operator from $H^{p}(\mathbb{R}^{n})$ into $L^{p} (\mathbb{R}^{n})$ if
and only if $T$ is bounded uniformly in $L^{p}$ norm on all $(p, \infty)$-atoms. A similar result is also obtained from $H^{p}(\mathbb{R}^{n})$ into $H^{p}(\mathbb{R}^{n})$.
\end{abstract}

\section{Introduction}

M. Bownik in \cite{Bow} gives an example of a linear functional defined on a dense subspace of Hardy space $H^{1}(\mathbb{R}^{n})$, which maps
all atoms into bounded scalars, but it can not be extended to a bounded functional on the whole space $H^{1}(\mathbb{R}^{n})$. That example is in certain sense pathological. In \cite{ricci}, F. Ricci and J. Verdera show that when $0 < p < 1$ no example like the one in \cite{Bow} can be exhibited. They also studied the extension problem for operators defined on the space of finite linear combinations of $(p, \infty)$-atoms, $0 < p < 1$, and taking values in a Banach space $B$, and proved that if the operator is uniformly bounded on $(p, \infty)$-atoms, then it extends to a bounded operator
from $H^{p}(\mathbb{R}^{n})$ into $B$.
\\
The extension property for operators defined on the space of finite linear combinations of $(1, q)$-atoms with $1< q < \infty$, and taking values in Banach spaces was studied by S. Meda, P. Sj\"ogren and M. Vallarino in \cite{meda}. For $0 < p \leq 1$ and $(p,2)$-atoms and operators taking values in quasi-Banach spaces, by D. Yang and Y. Zhou in \cite{yang}. \\
Y. Han and K. Zhao in \cite{han}, using the Calder\'on reproducing formula, give a $(p,2)$-atomic decomposition for members of $H^{p}(\mathbb{R}^{n}) \cap L^{2}(\mathbb{R}^{n})$, where the composition converges also in $L^{2}(\mathbb{R}^{n})$ rather than only in the distributions sense. Then, using this atomic decomposition, they prove that a bounded operator $T$ on $L^{2}(\mathbb{R}^{n})$ can be extended to a bounded operator from $H^{p}(\mathbb{R}^{n})$ into $L^{p} (\mathbb{R}^{n})$ if and only if $T$ is bounded uniformly on all $(p,2)$-atoms.\\
In this note we will show, by a simple argument, that if $f \in H^{p}(\mathbb{R}^{n}) \cap L^{s}(\mathbb{R}^{n})$ with $0 < p \leq 1 < s < \infty$, then the $(p, \infty)$-atomic decomposition given in \cite{stein}, (Theorem 2, p. 107), whose sum converges to $f$ in $H^{p}$ norm also converges to $f$ in $L^{s}$ norm. As a consequence of this, we prove that a bounded operator $T$ on $L^{s}(\mathbb{R}^{n})$ can be extended to a bounded operator from $H^{p}(\mathbb{R}^{n})$ into $L^{p} (\mathbb{R}^{n})$ if and only if $T$ is bounded uniformly in $L^{s}$ norm on all $(p, \infty)$-atoms. A similar result is also obtained from $H^{p}(\mathbb{R}^{n})$ into $H^{p}(\mathbb{R}^{n})$. See Theorem 5 and Corollary 6 and 7 below.

\section{Preliminaries}

For a given cube $Q$ in $\mathbb{R}^{n}$, we denote by $l(Q)$ its length. The following proposition gives the well known Whitney decomposition for an open nonempty proper subset of $\mathbb{R}^{n}$.

\begin{proposition} $($See p. $463$ in \cite{gra}$)$ Let $\mathcal{O}$ be an open nonempty proper subset of $\mathbb{R}^{n}$. Then there exists a family of closed cubes $\{ Q_k \}_{k}$ such that \\
\textbf{a)} $\bigcup_{k} Q_k = \mathcal{O}$ and the $Q_{k}$'s have disjoint interiors. \\
\textbf{b)} $\sqrt{n} \,\, l(Q_k) \leq dist\left( Q_{k}, \mathcal{O}^{c} \right) \leq 4 \sqrt{n} \,\, l(Q_{k}).$ \\
\textbf{c)} If the boundaries ot two cubes $Q_{i}$ and $Q_k$ touch, then
\begin{equation*}
\frac{1}{4} \leq \frac{l(Q_i)}{l(Q_k)} \leq 4.
\end{equation*}
\textbf{d)} For a given $Q_k$ there exist at most $12^{n}$ cubes $Q_{i}$'s that touch it.
\end{proposition}

\begin{remark} Let $\mathcal{F} = \{ Q_k \}$ be a Whitney decomposition of a proper open subset $\mathcal{O}$ of $\mathbb{R}^{n}$. Fix $0 < \epsilon < \frac{1}{4}$ and denote $Q_{k}^{*}$ the cube with the same center as $Q_k$ but with side length $(1+ \epsilon)$ times that $Q_k$. Then $Q_k$ touch $Q_i$ if and only if $Q_{k}^{*} \cap Q_{i}^{*} \neq \emptyset$. Consequently, every point in $\mathcal{O}$ is contained in at most $12^{n}$ cubes $Q_{k}^{*}$. Moreover, $0 < \epsilon < \frac{1}{4}$ can be chosen such that $\bigcup_k Q_{k}^{\ast} = \mathcal{O}$.
\end{remark}

\begin{remark} We consider $\mathcal{O}_1$ and $\mathcal{O}_2$ two open nonempty proper subset of $\mathbb{R}^{n}$ such that $\mathcal{O}_2 \subset \mathcal{O}_1$, and for $j=1,2$ let $\mathcal{F}_j = \{ Q^{j}_{k} \}$ be the Whitney decomposition of $\mathcal{O}_{j}$. Since $\mathcal{O}_2 \subset \mathcal{O}_1$, from Proposition 1-\textbf{b}, we obtain that if $Q_{i}^{2} \cap Q_{k}^{1} \neq \emptyset$ then
\begin{equation*}
l(Q_{i}^{2}) \leq 5 \,\, l(Q_{k}^{1}).
\end{equation*}
From this, it follows that for a given cube $Q_{i}^{2} \in \mathcal{F}_2$ there exist at most $7^{n}$ cubes in $\mathcal{F}_1$ such that $Q_{i}^{2} \cap Q_{k_j}^{1} \neq \emptyset$, for $j=1, ..., 7^{n}$ and $Q_{i}^{2} \subset \bigcup_{j=1}^{7^{n}} Q_{k_j}^{1}$; thus
$Q_{i}^{2 *} \subset \bigcup_{j=1}^{7^{n}} Q_{k_j}^{1 *}$. Finally, from Proposition 1-\textbf{d} and Remark 2, we obtain that there exist at most $84^{n}$ cubes $Q_{k}^{1 *}$'s that intersect to $Q_{i}^{2 *}$
\end{remark}

We conclude this preliminaries with the following

\begin{lemma} Let $ \{ \mathcal{O}_j \}_{j \in \mathbb{Z}}$ be a family of subset of $\mathbb{R}^{n}$ such that $\mathcal{O}_{j+1} \subset \mathcal{O}_j$ and
$\left| \bigcap_{j=1}^{\infty} \mathcal{O}_j \right| = 0$ . Then
\begin{equation*}
\sum_{j= - \infty}^{\infty} 2^{j} \chi_{\mathcal{O}_j}(x) \leq 2 \sum_{j= - \infty}^{\infty} 2^{j} \chi_{\mathcal{O}_j \setminus \mathcal{O}_{j+1}}(x),
\,\,\,\,\, p.p.x \in \mathbb{R}^{n}.
\end{equation*}
\end{lemma}

\begin{proof} We consider $N, M \in \mathbb{N}$. Now a computation gives
\begin{equation*}
\sum_{j=-N}^{M} 2^{j} \chi_{\mathcal{O}_j} = \frac{1}{2} \sum_{j=-N+1}^{M} 2^{j} \chi_{\mathcal{O}_j} + \sum_{j=-N}^{M-1} 2^{j} \chi_{\mathcal{O}_j \setminus \mathcal{O}_{j+1}} + 2^{M} \chi_{\mathcal{O}_M}
\end{equation*}
\begin{equation*}
\leq \frac{1}{2} \sum_{j=-N}^{M} 2^{j} \chi_{\mathcal{O}_j} + \sum_{j=-N}^{M-1} 2^{j} \chi_{\mathcal{O}_j \setminus \mathcal{O}_{j+1}} + 2^{M} \chi_{\mathcal{O}_M}.
\end{equation*}
Since $\lim_{M \rightarrow \infty} \chi_{\mathcal{O}_M} = \chi_{\bigcap_{j=1}^{\infty} \mathcal{O}_j}$ and
$\left| \bigcap_{j=1}^{\infty} \mathcal{O}_j \right| = 0$ the lemma follows.
\end{proof}

\qquad

\section{Main Result}

We recall the definition of Hardy space and $(p, \infty)$-atoms. \\
Define $\mathcal{F}_{N}=\left\{ \varphi \in S(\mathbb{R}^{n}):\sum\limits_{
\left\vert \mathbf{\beta }\right\vert \leq N}\sup\limits_{x\in \mathbb{R}
^{n}}\left( 1+\left\vert x\right\vert \right) ^{N}\left\vert \partial ^{
\mathbf{\beta }}\varphi (x)\right\vert \leq 1\right\} $. Denote by $\mathcal{M}$ the grand maximal operator given
by
\[
\mathcal{M}f(x)=\sup\limits_{t>0}\sup\limits_{\varphi \in \mathcal{F}
_{N}}\left\vert \left( t^{-n}\varphi (t^{-1}.)\ast f\right) \left( x\right)
\right\vert ,
\]
where $f\in S^{\prime}(\mathbb{R}^{n})$ and $N$ is a large and fix integer. For $0 < p < \infty$, the Hardy space
$H^{p}\left( \mathbb{R}^{n}\right) $ is the set of all $f\in S^{\prime }(\mathbb{R}^{n})$ for
which $\left\Vert \mathcal{M}f\right\Vert_{p}<\infty $. In this case we define $\left\Vert f\right\Vert
_{H^{p}}=\left\Vert \mathcal{M}f\right\Vert _{p}$. \\
For $1 < p < \infty$, it is well known that $H^{p}(\mathbb{R}^{n}) \cong L^{p}(\mathbb{R}^{n})$, $H^{1}(\mathbb{R}^{n}) \subset L^{1}(\mathbb{R}^{n})$ strictly, and for $0 < p < 1$ the spaces $H^{p}(\mathbb{R}^{n})$ and $L^{p}(\mathbb{R}^{n})$ are not comparable. One of the principal interest of $H^{p}(\mathbb{R}^{n})$ theory is that it gives a natural extension of the results for singular integrals, originally developed for $L^{p}$ ($p > 1$), to the range $0 < p \leq 1$. This is achieved to decompose elements in $H^{p}$ as sums of $(p,\infty)$-atoms.

\qquad

For $0 < p \leq 1$, a $(p,\infty)$-atom is a measurable function $a$ supported on a ball $B$ of $\mathbb{R}^{n}$ satisfying \\
$(i)$ $\| a\|_{\infty} \leq |B|^{-\frac{1}{p}}$, \\
$(ii)$ $\int x^{\alpha} a(x) dx = 0$, for all multiindex $\alpha$ with $|\alpha| \leq n(p^{-1}-1)$.

\qquad

Our main result is the following

\begin{theorem} Let $f \in H^{p}(\mathbb{R}^{n}) \cap L^{s}(\mathbb{R}^{n})$, with $0 < p \leq 1 < s < \infty$. Then there is a sequence of $(p, \infty)$-atoms $\{ a_j \}$ and a sequence of scalars $\{ \lambda_j \}$ with $\sum_{j} |\lambda_j |^{p} \leq c \| f \|_{H^{p}}^{p}$ such that $f = \sum_{j} \lambda_j a_j$, where the series converges to $f$ in $L^{s}(\mathbb{R}^{n})$.
\end{theorem}

\begin{proof} Given $f \in H^{p}(\mathbb{R}^{n}) \cap L^{s}(\mathbb{R}^{n})$, let $\mathcal{O}_j = \{ x : \mathcal{M}f(x) > 2^{j} \}$ and let $\mathcal{F}_j = \{ Q_{k}^{j} \}_{k}$ be the Whitney decomposition associated to $\mathcal{O}_j$ such that $\bigcup_{k} Q_{k}^{j \ast} = \mathcal{O}_j$. Following the proof 2.3 in \cite{stein} (p. 107-109), we have a sequence of functions $A_{k}^{j}$ such that

\qquad

(i) $supp(A_{k}^{j}) \subset Q_{k}^{j \ast} \cup \bigcup_{i \, :  \, Q_{i}^{j+1 \ast} \cap Q_{k}^{j \ast} \neq \emptyset} Q_{i}^{j+1 \ast}$ and
$|A_{k}^{j}(x)| \leq c 2^{j}$ for all $k,j \in \mathbb{Z}$.

\qquad

(ii) $\int x^{\alpha} A_{k}^{j}(x) dx =0$ for all $\alpha$ with $|\alpha| \leq n(p^{-1}-1)$ and all $k,j \in \mathbb{Z}$.

\qquad

(iii) The sum $\sum_{j,k} A_{k}^{j}$ converges to $f$ in the sense of distributions. \\
From (i) we obtain
$$\sum_{k} |A_{k}^{j}| \leq c 2^{j} \left( \sum_{k} \chi_{Q_{k}^{j \ast}} + \sum_{k} \chi_{\bigcup_{i \, :  \, Q_{i}^{j+1 \ast} \cap Q_{k}^{j \ast} \neq \emptyset} Q_{i}^{j+1 \ast}} \right)$$
the remark 2 gives
$$\leq c 2^{j} \left( \chi_{\mathcal{O}_j} + \sum_{k} \sum_{i \, :  \, Q_{i}^{j+1 \ast} \cap Q_{k}^{j \ast} \neq \emptyset} \chi_{Q_{i}^{j+1 \ast}} \right)$$
$$= c 2^{j} \left( \chi_{\mathcal{O}_j} + \sum_{i} \sum_{k \, :  \, Q_{i}^{j+1 \ast} \cap Q_{k}^{j \ast} \neq \emptyset} \chi_{Q_{i}^{j+1 \ast}} \right)$$
from remark 3 we have
$$\leq c 2^{j} \left( \chi_{\mathcal{O}_j} + 84^{n} \sum_{i} \chi_{Q_{i}^{j+1 \ast}} \right)$$
once again, from remark 2, we obtain
$$\leq c 2^{j} \left( \chi_{\mathcal{O}_j} + \chi_{\mathcal{O}_{j+1}} \right) \leq c 2^{j} \chi_{\mathcal{O}_j},$$
by Lemma 4 we can conclude that
$$\sum_{j,k} |A_{k}^{j}(x)| \leq c \sum_{j} 2^{j} \chi_{\mathcal{O}_j \setminus \mathcal{O}_{j+1}}(x), \,\,\,\, p.p.x \in \mathbb{R}^{n}.$$
Thus $$\int \left( \sum_{j,k} |A_{k}^{j}(x)| \right)^{s} dx \leq c \sum_j \int_{\mathcal{O}_j \setminus \mathcal{O}_{j+1}} 2^{js} dx \leq c \sum_j \int_{\mathcal{O}_j \setminus \mathcal{O}_{j+1}} (\mathcal{M}f(x))^{s} dx$$
$$\leq c \int_{\mathbb{R}^{n}} (\mathcal{M}f(x))^{s}  < \infty$$
since $f \in L^{s}(\mathbb{R}^{n})$. Now from (iii) we obtain that the sum $\sum_{j,k} A_{k}^{j}$ converges to $f$ in $L^{s}(\mathbb{R}^{n})$. \\
Finally, we set $a_{k}^{j} = \lambda^{-1}_{j,k} A_{k}^{j}$ with $\lambda_{j,k} = c 2^{j} |B_{k}^{j}|^{1/p}$, where $B_{k}^{j}$ is the smallest ball containing $Q_{k}^{j \ast}$ as well as all the $Q_{i}^{j+1 \ast}$ that intersect $Q_{k}^{j \ast}$. Then we have a sequence $\{ a_{j,k} \}$ of $(p, \infty$)-atoms and a sequence of scalars $\{ \lambda_{j,k} \}$ such that the sum $\sum_{j,k} \lambda_{j,k} a_{j,k}$ converges to $f$ in $L^{s}(\mathbb{R}^{n})$ with $\sum_{j,k} |\lambda_{j,k} |^{p} \leq c \| f \|_{H^{p}}^{p}$ (see (38), in \cite{stein} p. 109). The proof of the theorem is therefore concluded.
\end{proof}

\begin{corollary} Let $T$ be a bounded operator on $L^{s}(\mathbb{R}^{n})$ for some $1< s < \infty$. Then $T$ can be extended to a bounded operator from $H^{p}(\mathbb{R}^{n})$ into $L^{p} (\mathbb{R}^{n})$ $(0<p\leq1)$ if and only if $T$ is bounded uniformly in the $L^{p}$ norm on all $(p, \infty)$-atoms.
\end{corollary}

\begin{proof} Since $T$ is a bounded operator on $L^{s}(\mathbb{R}^{n})$, $T$ is well defined on $H^{p}(\mathbb{R}^{n}) \cap L^{s}(\mathbb{R}^{n})$ for $0 < p \leq 1$. If $T : H^{p}(\mathbb{R}^{n}) \cap L^{s}(\mathbb{R}^{n}) \rightarrow L^{p}(\mathbb{R}^{n})$ can be extended to a bounded operator from $H^{p}(\mathbb{R}^{n})$ into $L^{p} (\mathbb{R}^{n})$, then $\| Ta \|_{p} \leq c_p \|a \|_{H^{p}}$ for all $(p, \infty)$-atom $a$. Since there exists a universal constant $C$ such that $\|a \|_{H^{p}} \leq C < \infty$ for all $(p, \infty)$-atom $a$; it follows that $\| Ta \|_{p} \leq C_p $ for all $(p, \infty)$-atom $a$. \\
Reciprocally, let $f \in H^{p}(\mathbb{R}^{n}) \cap L^{s}(\mathbb{R}^{n})$ by Theorem 5 there is a $(p, \infty)$-atom decomposition such that $\sum_{j} \lambda_j a_j = f$ in $L^{s}(\mathbb{R}^{n})$. Since $T$ is bounded on $L^{s}(\mathbb{R}^{n})$ we have that the sum $\sum_j \lambda_j Ta_j$ converges a $Tf$ in $L^{s}(\mathbb{R}^{n})$, thus there exists a subsequence of natural numbers $\{j_N \}_{N \in \mathbb{N}}$ such that $\lim_{N \rightarrow \infty} \sum_{j=-j_N}^{j_N} \lambda_{j} Ta_{j}(x)= Tf(x)$ p.p.$x \in \mathbb{R}^{n}$, this implies
$$|Tf(x)| \leq \sum_j |\lambda_j Ta_j(x)|, \,\,\,\, p.p.x \in \mathbb{R}^{n}.$$
If $\| Ta \|_{p} \leq C_p < \infty $ for all $(p, \infty)$-atom $a$, and since $0< p \leq 1$ we get
$$\|Tf \|_{p}^{p} \leq \sum_j |\lambda_j|^{p} \|Ta_j \|_{p}^{p} \leq C_{p}^{p} \sum_j |\lambda_j|^{p} \leq C_{p}^{p} \|f \|_{H^{p}}^{p}$$
for all $f \in H^{p}(\mathbb{R}^{n}) \cap L^{s}(\mathbb{R}^{n})$. Since $H^{p}(\mathbb{R}^{n}) \cap L^{s}(\mathbb{R}^{n})$ is a dense subspace of $H^{p}(\mathbb{R}^{n})$, the corollary follows.
\end{proof}
Since many operators that appear in the practice are bounded on $L^{s}$ $(1< s < \infty)$, in view of Corollary 6 is sufficient to prove that our operator is bounded uniformly in the $L^{p}$ norm on all $(p, \infty)$-atoms to assure the boundedness $H^{p}-L^{p}$.

\begin{corollary} Let $T$ be a bounded operator on $L^{s}(\mathbb{R}^{n})$ for some $1< s < \infty$. Then $T$ can be extended to a bounded operator on $H^{p}(\mathbb{R}^{n})$ $(0<p\leq1)$ if and only if $T$ is bounded uniformly in the $H^{p}$ norm on all $(p, \infty)$-atoms.
\end{corollary}

\begin{proof} The only if is similar to the only if of the corollary 6. By the other hand, once again by Theorem 5, for $f \in H^{p}(\mathbb{R}^{n}) \cap L^{s}(\mathbb{R}^{n})$,  we have a $(p,\infty)$-atomic decomposition such that
 $\sum_j \lambda_j a_j = f$ in $L^{s}$ norm. Since $T$ is bounded on $L^{s}$ we obtain
$$|\mathcal{M}(Tf)(x)| \leq \sum_j |\lambda_j \mathcal{M}(Ta_j)(x)|, \,\,\,\, p.p.x \in \mathbb{R}^{n}.$$
Now, if $\| Ta \|_{H^{p}} \leq C_p < \infty$ for all $(p, \infty)$-atom $a$ with $0 < p \leq 1$, we get
\begin{equation*}
\|Tf \|_{H^{p}}^{p} = \|\mathcal{M}(Tf) \|_{p}^{p} \leq \sum_j |\lambda_j|^{p} \| \mathcal{M}(Ta_j) \|_{p}^{p} \leq C_{p}^{p} \sum_j |\lambda_j|^{p}
\leq C_{p}^{p} \| f \|_{H^{p}}^{p},
\end{equation*}
for all $f \in H^{p}(\mathbb{R}^{n}) \cap L^{s}(\mathbb{R}^{n})$. Since $H^{p}(\mathbb{R}^{n}) \cap L^{s}(\mathbb{R}^{n})$ is a dense subspace of $H^{p}(\mathbb{R}^{n})$, the corollary follows.

\end{proof}

\begin{remark} The statement of Theorem 5 appears in \cite{nakai}, as well as (of implicity way) Lemma 4 (see p. 3692).
\end{remark}

\end{document}